\documentclass[11pt, a4paper,reqno]{amsart}

\usepackage[T1]{fontenc}

\usepackage{amsmath}
\usepackage{amssymb}
\usepackage{amsthm}
\usepackage{graphicx}
\usepackage{multirow}
\usepackage{centernot}
\usepackage{accents}

\usepackage{mathtools}

\usepackage{tikz}
\usepackage{tikz-cd}
\usetikzlibrary{shapes}

\usepackage[cal=cm,calscaled=.96]{mathalpha}

\newcommand{\N}{\mathbb{N}}
\newcommand{\Z}{\mathbb{Z}}
\newcommand{\R}{\mathbb{R}}
\newcommand{\A}{\ensuremath{\mathcal{A}}}

\newcommand{\D}{\ensuremath{\mathcal{D}}}
\newcommand{\Q}{\ensuremath{\mathcal{Q}}}
\newcommand{\IE}{\ensuremath{\Q \xrightarrow{\;\;E\;\;} \D}}

\DeclareMathOperator{\dual}{op}

\usepackage{listings}
\usepackage[margin=3cm]{geometry}

\newtheorem{thm}{Theorem}[section]
\newtheorem{nota}[thm]{Notation}
\newtheorem{defn}[thm]{Definition}
\newtheorem{fact}[thm]{Fact}

\newtheorem{cor}[thm]{Corollary}

\usepackage[colorlinks, linkcolor=black,citecolor=black,linktocpage,hypertexnames=true,pdfpagelabels]{hyperref}

\begin{document}

\title[Dose-Escalation Designs Extending to Admit Titration]{Dose-Escalation Trial Protocols that\\ Extend Naturally to Admit Titration}
\author{David C. Norris}
\address{Precision Methodologies, LLC,  Wayland, Massachusetts}
\email{david@precisionmethods.guru}

\date{July 1, 2025}

\subjclass[2020]{92C50, 18B35, 18A40}

\thanks{\emph{Acknowledgments.} Mark Thom's Scryer Prolog, and the community
  that has grown around it, have made available a free, extensible,
  Standard-conforming Prolog system crucial for the computational aspects of
  this work.  Markus Triska has illuminated Prolog for me from a modern
  standpoint, and continually provokes improvements in my code and my thinking.
  An anonymous reviewer generously dissected an earlier version of this work
  presented at the Eighth International Conference on Applied Category Theory.
  The remaining errors are mine.  I dedicate this work to my father, and to the
  memory of my mother.}

\begin{abstract}%
Dose-escalation trials in oncology drug development still today typically aim to
identify 1-size-fits-all dose recommendations, as arbitrary quantiles of the
toxicity thresholds evident in patient samples.  In the late 1990's efforts to
individualize dosing emerged fleetingly in the oncology trial methods
literature, but these have gained little traction due to a nexus of conceptual,
technical, commercial, and regulatory barriers.  To reduce the `activation
energy' needed for transforming current 1-size-fits-all dose-escalation trial
designs to the dose-{\em titration} designs required for patient-centered dose
individualization, we demonstrate a categorical formulation of dose-escalation
protocols that extends readily to allow gradual introduction of dose titration.

Central to this formulation is a symmetric monoidal preorder on the accessible
states of dose-escalation trials, embodying pharmacologic intuitions regarding
dose-monotonicity of drug toxicity and ethical intuitions relating to the
therapeutic intent of such trials.  A trial protocol that assigns doses to
sequentially enrolled participants consistently with these intuitions is then a
monotone map from this preorder to the ordered finite set of doses being
trialed.  We illustrate this formulation by reference to the ubiquitous `3+3'
dose-escalation design, which despite its many widely discussed flaws remains
familiar to oncology trialists and moreover has available an executable
specification in Prolog.  Remarkably, examined in light of our preorder the 3+3
protocol discloses a {\em new flaw} not previously described: a non-monotone
dose recommendation.  The right Kan extension approximates this protocol from
the side of safety, dissolving its 3-at-a-time cohorts to allow {\em
  incremental} enrollment, and perforce rectifying said non-monotonicity.  It
also facilitates accelerated enrollment while toxicity assessments remain pending,
and indeed discretionary dose titration as well.

A basic simulation experiment is presented, demonstrating the feasibility of
trial protocols incorporating these elements, built on the right Kan extension
as well as a strictly safer and more parsimoniously parametrized lower-Galois
enrollment derived from it.  Further efforts along these lines might aim to
approximate any of several more modern dose-escalation designs that have begun
to supplant the 3+3, or seek de novo designs with specified safety properties
within the finite (if large) spaces of lower-Galois enrollment functors.
\end{abstract}

\maketitle

\setcounter{tocdepth}{1}
\vspace{-10mm}
\tableofcontents

\section{Introduction}

Despite the long-recognized heterogeneity in patients' pharmacokinetics and pharmacodynamics \cite{edler-statistical-1990}, dose-finding trials in oncology still today generally aim to identify a 1-size-fits-all dose recommendation, in the form of an arbitrary quantile of the population distribution of toxicity thresholds \cite{norris-comment-2020}.  While glimmers of hope have appeared fleetingly in the trial methods literature \cite{daugherty-learning-1997,simon-accelerated-1997}, a complex interaction of conceptual, technical and political factors \cite{norris-dose-2017,norris-ethical-2019,norris-how-2023} has impeded progress toward dose individualization in oncology trials.

Like an enzyme that reduces the activation energy needed for a chemical reaction, the present work aims to catalyze the transformation of existing dose-escalation designs to dose-{\em titration} designs.  A categorical formulation yields a standpoint from which this transformation may be conceived and implemented naturally, instead of being regarded as a special or problematical case.

\section{Operational Details of Dose-Escalation Trials}\label{ops}

In the methods literature, dose-finding designs are often defined and analyzed in highly stylized settings that abstract away certain essential operational details inherent to actual trials.  Adapting such designs --- often defined in terms of non-unit-sized `cohorts' enrolled in discrete time --- to continuous-time trial operations may be difficult.  Grafting on such considerations {\em after the fact} may involve massive (yet still incomplete) protocol tabulations \cite{skolnik-shortening-2008,frankel-model-2020}.  Here, however, we model these explicitly from the outset:

In queueing-theory terms, a dose-finding protocol amounts to a {\em service policy} for an {\em arrivals process} in which patients who have cancer, after exhausting standard treatment options, express willingness to try an experimental treatment.  At any time $\tau$, the protocol specifies a \textbf{current enrolling dose [level]} $d(\tau) \in \{0,1,...,D\}$ indexing an increasing sequence of dose intensities $\{0=x_0 < x_1 < ... < x_D\}$ that were fixed {\em ex ante}.  Patients arriving at a time $\tau$ when $d(\tau) = 0$ are placed into a waiting queue.  While $d(\tau) > 0$ any waiting patients are enrolled in order of arrival,\footnote{Since enrollment may generally update the current enrolling dose, preserving $d(\tau)$ as a well-defined function of time requires that enrolling each individual from a non-empty waiting queue takes some nonzero time interval, albeit one that may be treated as effectively infinitesimal.} and if the waiting queue is empty new arrivals are enrolled promptly.

\begin{nota}\label{assessment}
  The \textbf{participants} in a dose-escalation trial, indexed by $i \in I$, \textbf{enroll} at time $\tau_0^i$ into dose level $d^i$.  Given that toxic responses generally manifest with some latency after dose administration, \textbf{toxicity assessment} remains \textbf{pending} for participant $i$ until some time $\tau_1^i \in (\tau_0^i, \tau_0^i + \delta \tau]$ when the assessment \textbf{resolves} into one of three \textbf{outcomes}:
  
    \begin{itemize}
    \item Participant $i$ is found to have experienced an (intolerable) \textbf{toxicity},
    \item to have become \textbf{inevaluable} due to early withdrawal from the trial or death unrelated to toxicity,
    \item or otherwise (at $\tau_1^i = \tau_0^i + \delta\tau$) is assessed to have \textbf{tolerated} their dose.\footnote{We here ignore late-manifesting toxicities that occur after the lapse of time $\delta\tau$, which is typically on the order of 1 month.}
    \end{itemize}
\end{nota}

\begin{nota}\label{outcome-notation}
  We indicate \textbf{evaluability} by $n^i \in \{0,1\}$, and occurrence of \textbf{toxicity} by $y^i \in \{0,1\}$ or sometimes more distinctively by $y^i \in \{\textup{\texttt{o}},\textup{\texttt{x}}\}$.
\end{nota}

\begin{nota}
  We write $I_d(\tau) \subseteq I$ for the subset of individuals enrolled at dose $d$ whose assessments have resolved by time $\tau$:
  $$
  I_d(\tau) = \{ i \in I \mid d^i = d,\, \tau_1^i \le \tau \}.
  $$
\end{nota}

\begin{defn}
  For each dose $d$, we have at any time $\tau$ the formal quotient,
  $$
  q_d(\tau) = \frac{t_d}{n_d}(\tau) = \sum_{i \in I_d(\tau)} \frac{y^i}{n^i} \;\in\; Q = \{ t/n \mid t, n \in \N;\; t \le n \},
  $$
  recording the assessment of $t_d$ toxic responses among $n_d$ evaluable trial participants who have received dose $d$.  The vector of quotients $q(\tau) = (q_1,...,q_D)(\tau) \in Q^D$ will be called the \textbf{cumulative tally} at time $\tau$.  Together with the \textbf{pending count} $p : \R^+ \rightarrow \N^D$ having dosewise components
  $$
  p_d(\tau) = | \{ i \in I \mid d^i = d,\, \tau_0^i \le \tau < \tau_1^i \} |,
  $$
  this constitutes the \textbf{enrolled state}, $s(\tau) = (q(\tau), p(\tau))$.
\end{defn}

\begin{nota}
  By implicitly regarding $\tau$ as an arbitrary `current time' or `now', we will often freely suppress the $\tau$-dependence of $I_d$, $q$, $p$ and $s$.
\end{nota}

In order to determine the enrolling dose $d(\tau)$ as a function of $s(\tau)$, we proceed to elaborate a symmetric, monoidal partial order on $Q^D$ that captures certain fundamental pharmacologic and ethical intuitions.

\section{Modeling Pharmacologic Monotonicities}

\begin{defn}
  Let $+:Q \times Q \rightarrow Q$ be defined by
  $$
  \frac{t_1}{n_1} + \frac{t_2}{n_2} = \frac{t_1 + t_2}{n_1 + n_2}.
  $$
  Observe that this is a monoidal operation with unit $\frac{0}{0}$, which extends in the obvious way to a monoidal operation on $Q^D$ with unit $(\frac{0}{0},...,\frac{0}{0})$.
\end{defn}

\begin{defn}\label{lesaf}
  Let $\preceq$ be the transitive closure of a preorder relation satisfying,
  \begin{equation}
  \frac{t}{n}\!+\!\frac{1}{1} \;\preceq\; \frac{t}{n} \;\preceq\; \frac{t}{n}\!+\!\frac{0}{1} \quad \forall\; \frac{t}{n} \in Q. \label{lesafcond}
  \end{equation}
  Then the preorder $(Q,\preceq)$ compares the \textbf{evident safety} expressed in dosewise tallies, such that we read
  $$
  q_1 \preceq q_2
  $$
  as ``$q_1$ is evidently no safer than $q_2$'' or ``$q_2$ is evidently at least as safe as $q_1$''.
\end{defn}

\begin{fact}
  $(Q,\preceq,\frac{0}{0},+)$ is a symmetric monoidal preorder.  It is easy to see that $+$ is a symmetric monoidal operation on $Q$ with unit $0/0$, the necessary unitality, associativity and commutativity all being inherited directly from the monoid $(\N,0,+)$.  The monotonicity condition.
  $$
  q \preceq q',\; g \preceq g' \implies q + g \preceq q' + g',
  $$
  arises by induction from the Definition~\ref{lesaf} of $\preceq$ in terms of $+$.
\end{fact}

\begin{fact}\label{charQ1}
  $$
  \frac{t}{n} \preceq \frac{t'}{n'} \quad \iff \quad t \ge t' + \max(0,n-n').
  $$
\end{fact}
\begin{proof}
  This is most easily seen by expressing \eqref{lesafcond} in its equivalent {\em ratio} form,
  $$
  t\!:\!u + 1\!:\!0 \preceq t\!:\!u \preceq t\!:\!u + 0\!:\!1 \quad \forall\; t\!:\!u \equiv \frac{t}{t+u} \in Q,
  $$
  and observing that consequently $t\!:\!u \preceq t'\!:\!u'$ iff $t \ge t'$ and $u \le u'$.  This latter condition, in turn, may be transformed as follows:
  \begin{align*}
    & t \ge t' \;\wedge\; u \le u' \\
    \iff\quad & t \ge t' \;\wedge\; n-t \le n'-t' \\
    \iff\quad & t \ge t' \;\wedge\; t \ge t' + (n-n') \\
    \iff\quad & t \ge t' + \max(0,n-n').
  \end{align*}
\end{proof}

\begin{nota}
  Let $\langle q \rangle_j$ denote the tally $(\frac{0}{0},...,\frac{0}{0},q,\frac{0}{0},...,\frac{0}{0}) \in Q^D$ with $q \in Q$ in the $j$'th position and $0/0$ elsewhere, and let $\langle q, q' \rangle_{j,k}$ denote the tally $\langle q \rangle_j + \langle q' \rangle_k$ with $q, q' \in Q$ in the $j$'th and $k$'th positions of an otherwise-$0/0$ tally.  It is to be understood that $j < k$ whenever this latter notation is used.
\end{nota}

\begin{nota}
  The sheer fact of having recorded a tally of the form $\langle\frac{1}{1},\frac{0}{1}\rangle_{j,k}$ means that we enrolled participants $i, i' \in I$ at doses $x_j < x_k$ respectively, and upon assessment found that:
  $$
  y(i,x_j) = 1, \; y(i',x_k) = 0,
  $$
  in which we have expanded the \underline{observed} $y^i$ of Notation~\ref{outcome-notation} to a monotone function $y(i,-):\R^+\rightarrow\{0,1\}$ that embraces the \underline{counterfactual} (or `potential') outcomes for individual $i$ at any hypothetical dose.
\end{nota}
  
Thus we may regard the {\em possibility of observing} $\langle\frac{1}{1},\frac{0}{1}\rangle_{j,k}$ as \textbf{equivalent to a proposition:}
  $$
  \langle\frac{1}{1},\frac{0}{1}\rangle_{j,k} \; \equiv \; \exists\, i, i' \in I \;\;\mbox{such that}\;\;\, y(i,x_j) = 1 \;\mbox{and}\;\, y(i',x_k) = 0.
  $$
  On this understanding, we can express the pharmacologic premise of \textbf{monotone dose-toxicity} via,
  $$
  \langle\frac{1}{1}\rangle_j \implies \langle\frac{1}{1}\rangle_k \;\forall\, j < k
  \quad\quad\text{and}\quad\quad
  \langle\frac{0}{1}\rangle_j \impliedby \langle\frac{0}{1}\rangle_k \;\forall\, j < k.
  $$
In words: a participant who experiences a toxicity at dose $j$ would also experience a toxicity at any higher dose $k>j$; conversely, a participant who tolerates dose $k$ would also tolerate any lower dose $j<k$.
\begin{defn}\label{domon}
  Let $\preceq_0$ denote the monoidal preorder relation on $Q^D$ generated by the following arrows:
  $$
  \begin{aligned}
    \langle\frac{0}{0}\rangle      &\preceq_{tol_1} \langle\frac{0}{1}\rangle_1 \\
    \langle\frac{0}{1}\rangle_{j-1} &\preceq_{titro_j} \langle\frac{0}{1}\rangle_j ,\; j \in \{2,...,D\} \\
    \langle\frac{1}{1}\rangle_{j} &\preceq_{titrx_j} \langle\frac{1}{1}\rangle_{j+1} ,\; j \in \{1,...,D-1\} \\
    \langle\frac{1}{1}\rangle_D &\preceq_{det_D} \langle\frac{0}{0}\rangle
  \end{aligned}
  $$
  We call a monoidal preorder relation $\preceq$ on $Q^D$ \textbf{dose-monotone} iff $\preceq_0\,\subseteq\,\preceq$.
\end{defn}

\begin{nota}
  Without ambiguity, we let each of the designations $\preceq_*$ of Definition~\ref{domon} stand for any relation which it implies directly by monoidality.  Thus, we write simply $q \preceq_{tol_1} q'$ whenever $q+\langle\frac{0}{1}\rangle_1 = q'$, we write $q \preceq_{titro_j} q'$ to mean $\exists b \in Q^D$ such that $q = b + \langle\frac{0}{1}\rangle_{j-1} \preceq_{titro_j} \langle\frac{0}{1}\rangle_j + b = q'$, and so forth.
\end{nota}
  
The subscripts on $\preceq_*$ in Definition~\ref{domon} indicate the underlying intuitions of these `atomic' arrows, considered as incremental transformations which tallies may undergo as a trial progresses.  Thus, observing a new participant's toleration of dose 1 yields a new tally that is evidently safer:
$$
q \preceq_{tol_1} q + \langle\frac{0}{1}\rangle_1.
$$
Conversely, the observation of a new toxicity --- even at the highest dose, where it is least surprising --- yields a tally that is evidently less safe:
$$
q + \langle\frac{1}{1}\rangle_D \preceq_{det_D} q.
$$

The transformations corresponding to $\preceq_{titro_j}$ and $\preceq_{titrx_j}$ would be (respectively) those in which a trial participant tolerates dose $j$ after titrating upward from a tolerated dose $j-1$, and where a participant experiences toxicity at dose $j$ after a dose reduction from an intolerable dose $j+1$.\footnote{Interpreted thus, the arrows of $\preceq_{det_D}$ and $\preceq_{titrx_j}$ run opposite the `arrow of time'.}  That dose-{\em escalation} designs exclude such titration maneuvers {\em by definition}\footnote{I am appealing here to an escalation--titration distinction introduced in \cite{norris-dose-2017} with some support from the treatment of these issues in e.g. \cite{senn-statistical-2007-1}.} does not exempt them from the underlying pharmacological principle expressed in these arrows.  Thus, we are entitled to examine dose-escalation trials in light of this idea, even if their designs overtly ignore it.

To see why any sensible preorder $\preceq$ on $Q^D$ must be {\em monoidal}, imagine that a dose-escalation study is being conducted at two different medical centers.  The investigators at center A notice that, if they break out their own current tally by sex, $q^A = q^A_f + q^A_m$, they find $q^A_m \preceq q^A_f$ --- the drug looks less toxic in females.  Meanwhile, center B investigators have noticed the same phenomenon locally: $q^B_m \preceq q^B_f$.  Monoidality ensures this finding does not paradoxically vanish upon pooling the data: $q_m = q^A_m+q^B_m \preceq q^A_f+q^B_f = q_f$.

\subsection{An explicit characterization of $\preceq_0$}

For the $D=3$ case, we can depict the atomic transformations $q = (\frac{t_1}{n_1},\frac{t_2}{n_2},\frac{t_3}{n_3}) \preceq_a (\frac{t_1'}{n_1'},\frac{t_2'}{n_2'},\frac{t_3'}{n_3'}) = q'$ of Definition~\ref{domon} as follows:
  \begin{center}
  \begin{tikzpicture}
    \matrix (mat) [%
      matrix of math nodes,
      column 1/.style={anchor=east},
      column 8/.style={anchor=west}]
            {%
    & u_1 & u_2 & u_3 & t_1 & t_2 & t_3 \\
    (\preceq_{tol_1})    & +1 &    &    &    &    &    \\ 
    (\preceq_{titro_2})  & -1 & +1 &    &    &    &    \\ 
    (\preceq_{titro_3})  &    & -1 & +1 &    &    &    \\ 
    (\preceq_{titrx_1})  &    &    &    & -1 & +1 &    \\ 
    (\preceq_{titrx_2})  &    &    &    &    & -1 & +1 \\ 
    (\preceq_{det_3})    &    &    &    &    &    & -1 \\ 
    & u_1' & u_2' & u_3' & t_1' & t_2' & t_3' \\
            };
  \end{tikzpicture}    
  \end{center}
  Note the cascading effect here, in which $\mathrm{tol}_1$ arrows inject \texttt{o}'s at the lowest dose and the $\mathrm{titro}_d$ titrate these upward, all without affecting the toxicity counts $t_d$; whereas the $\mathrm{titrx}_d$ and $\mathrm{det}_D$ conspire to shift \texttt{x}'s upward and exit stage right.  Observing these `flows' may help to motivate the following Definitions:


  \begin{defn}\label{doseintensity}
    The \textbf{dose intensity}\footnote{This term already enjoys widespread use in oncology, with which this definition is concordant.} of a dosewise vector of counts $(c_d)_{d=1}^D\in\N^D$ is the vector,
    $$
    C_d = (\sum_{j=d}^D c_j),\;1\le d\le D.
    $$
    For a tally $(t/n)=(t\!:\!u)\in Q^D$, we will speak of the \textbf{tolerated dose intensity} and \textbf{[net] dose intensity},
    $$
    U_d = (\sum_{j=d}^D u_j)
    \quad\mathrm{and}\quad
    N_d = (\sum_{j=d}^D n_j).
    $$
    Being a sequence of `upper tails', a dose intensity is decreasing---much like a survival curve.
  \end{defn}
  
  \begin{defn}\label{toxprofile}
    The \textbf{toxicity profile} of a tally $\frac{t}{n}\in Q^D$ is the distribution,
    $$
    T_d = (\sum_{j=1}^d t_j).
    $$
    Being a sequence of `lower tails', the $(T_d)$ are increasing---like a cumulative distribution function.
  \end{defn}

  \begin{fact}\label{enst0}
    $q \preceq_0 q' \iff U\le U'$ and $T\ge T'$.
  \end{fact}
  \begin{proof}
    Since each atomic arrow of Definition~\ref{domon} preserves $U\le U'$ and $T\ge T'$, so must their transitive closure.  Conversely, by considering the $u_d$'s and $t_d$'s as {\em labeled individuals}, it is readily seen that any transformation of $u$'s preserving $U\le U'$ is composable by \texttt{o}'s entering at left then shifting rightward, while any transformation of $t$'s preserving $T\ge T'$ is composable as right-shifts of \texttt{x}'s followed by exits at the right.
  \end{proof}
  Thus we have the intuitively appealing interpretation that $q\preceq_0q'$ iff $q'$ has at least as much tolerated dose intensity as $q$, without a raised toxicity profile.

  \begin{nota}
    Let $Q^D \xrightarrow{\sigma} \N^D\times\N^D$ denote the mapping $(t\!:\!u) \mapsto (T,U)$.
  \end{nota}

  \begin{nota}
     Let $\Q_0$ denote the symmetric monoidal preorder, $(Q^D,\preceq_0,\langle\frac{0}{0}\rangle,+)$.
  \end{nota}

  \begin{fact}
    $\Q_0 \xhookrightarrow{\sigma} (\N^D,\le)^{\dual}\times(\N^D,\le)$ is a monotone embedding.
  \end{fact}
  \begin{proof}
    Differencing $T$ and $U$ recovers $t$ and $u$, so $\sigma$ is injective.  Monotonicity follows from Fact~\ref{enst0}.
  \end{proof}

  \begin{cor}
    $\Q_0$ is a lattice, since the image $\sigma\Q_0$ is closed under meets and joins.
  \end{cor}

\subsection{Therapeutic intent}

The similarity of Fact~\ref{enst0} to Fact~\ref{charQ1} suggests that $(Q^D,\preceq_0,\langle\frac{0}{0}\rangle,+)$ generalizes $(Q,\preceq,\frac{0}{0},+)$ in a natural way to multiple doses.  But this generalization proves insufficient for modeling of dose-escalation trials, and requires strengthening by the recognition of additional principles.

\begin{defn}\label{therapeutic}
  Let $\preceq_1\;\supset\;\preceq_0$ denote the dose-monotone preorder relation on $Q^D$ generated upon $\preceq_0$ by including as well the following arrows:
  $$
  \begin{aligned}
    \langle\frac{1}{2}\rangle_D &\preceq_{bal} \langle\frac{0}{0}\rangle \\
    \langle\frac{1}{1},\frac{0}{1}\rangle_{j,k} &\preceq_{exch_{j,k}} \langle\frac{0}{1},\frac{1}{1}\rangle_{j,k} ,\; 1\le j<k\le D.
  \end{aligned}
  $$
  We call such a preorder \textbf{therapeutic}, for reasons to be elaborated presently.
\end{defn}

The $\preceq_{bal}$ arrows serve to break a symmetry that would otherwise exist between observed toxicity and tolerability.  They state that observed in a $1\!:\!1$ ratio at the highest dose, toxicity and non-toxicity on balance yield a less safe tally.  Intuitively, we might understand these judgments as establishing a prior expectation of toxicity rate below 0.5 {\em even at the highest dose},\footnote{This expectation becomes binding for any given dose at the time when a decision is made to {\em enroll} patients at that dose.  Upon trial initiation, this commitment is necessary only regarding the lowest dose.} so that the derogatory informational content (entropy) of a toxicity outweighs the favorable information in a non-toxicity.

Similar considerations help us to understand the $\preceq_{exch_*}$ also as breaking a toxicity--tolerability symmetry, albeit now {\em across two distinct doses}.  To appreciate the $\preceq_{exch_*}$ arrows, pick any two doses $x_1 < x_2 \in \R^+$ and consider them using the $D=2$ case of our notation.\footnote{The basic pharmacologic intuitions we aim to elicit here logically precede any such concrete details of trial design as our pre-specification of $D$ doses $x_1,...,x_D$.}  Suppose we sample pairs $(i,i')$ of distinct individuals from a population with a continuously distributed latent toxicity threshold, assigning $i$ to receive dose 1 and $i'$ to receive dose 2.  Then observing $(\tt{x},\tt{o})$ means that individual $i$ experienced toxicity at dose 1 while $i'$ tolerated dose 2.  Due to the monotonicity of dose-response, we then know that (counterfactually) had we sampled these individuals in the opposite order $(i',i)$, we would have observed $(\tt{o},\tt{x})$.  Thus each observed $(\tt{x},\tt{o})$ points to {\em an ensemble of potential samples} in which $(\tt{x},\tt{o})$ and $(\tt{o},\tt{x})$ observations match one-to-one.  But crucially, no such implication arises in the opposite direction, from an observation of $(\tt{o},\tt{x})$.  Consequently, there is a sense in which
$$
(\tt{x},\tt{o}) \implies (\tt{o},\tt{x}),
$$
so that we may say $(\tt{x},\tt{o})$ has {\em higher information content} than $(\tt{o},\tt{x})$.\footnote{Consider for example that observing $(\tt{x},\tt{o})$ absolutely excludes the possibility that dose 1 might be completely nontoxic, whereas $(\tt{o},\tt{x})$ excludes only the stronger claim that {\em both} doses are completely nontoxic.}  Provided that we chose both doses (and in particular, the higher $x_2$) with primarily {\em therapeutic intent},\footnote{See e.g. \cite{weber-reaffirming-2016} and \cite{burris-correcting-2019}, which elaborate the doctrine of {\em therapeutic intent} in early-phase cancer clinical trials.} which requires a prior expectation of toxicity substantially below 0.5, then both $(\tt{x},\tt{o})$ and $(\tt{o},\tt{x})$ must be seen to have net {\em derogatory} content regarding evident safety.  Thus, the stronger $(\tt{x},\tt{o})$ is the {\em more derogatory} of the two:
$$
\langle\frac{1}{1},\frac{0}{1}\rangle \preceq_{exch_{12}} \langle\frac{0}{1},\frac{1}{1}\rangle.
$$

\subsection{A sequence of nested preorders $\preceq_r$}

As we will see shortly, $\preceq_{bal}$ proves to be a somewhat weak condition which we can profitably strengthen in a graded manner, as follows:
\begin{nota}
  For any $r\in \N^+$, let $\preceq_{bal_r}$ denote the monoidal arrows generated by,
  $$
  \langle\frac{1}{1+r}\rangle_D \preceq_{bal_r} \langle\frac{0}{0}\rangle.
  $$
\end{nota}

\begin{fact}\label{rimplies}
  $q \preceq_{bal_r} q' \implies q \preceq_{bal} q'$.
\end{fact}
\begin{proof}
  Observe that $q \preceq_{bal_r} q' \implies q \preceq_{bal_{r-1}} q'$:
  $$
  \langle\frac{0}{0},...,\frac{1}{r}\rangle \preceq_{tol_1} \langle\frac{0}{1},...,\frac{1}{r}\rangle \preceq_{titro_2}...\preceq_{titro_D} \langle\frac{0}{0},...,\frac{1}{1+r}\rangle \preceq_{bal_r} \langle\frac{0}{0}\rangle,
  $$
  allowing recursion on $r$ to the base case  $\preceq_{bal_1} \equiv \preceq_{bal}$.
\end{proof}

\begin{defn}\label{lesaf^D}
  For $r \in \N^+$, let $\preceq_r$ denote the monoidal preorder relation $(\preceq_1)\cup\{\preceq_{bal_r}\}$ on $Q^D$.
\end{defn}

\begin{nota}
  Let $\Q_r$ denote $(Q^D,\preceq_r,\langle\frac{0}{0}\rangle,+)$.
\end{nota}

\begin{fact}
  The $(\Q_r)_{r\in\N}$ form a nested sequence of subcategories, $\Q_r \hookrightarrow \Q_{r+1}.$
\end{fact}

\subsection{An explicit characterization of $\preceq_r$}

Note that it is {\em almost} the case that $\preceq_{exch_{12}}\preceq_{exch_{23}}\, = \,\preceq_{exch_{13}}$.  We have for example that
$$
\begin{tikzcd}[row sep=large,column sep=small]
  (\frac{1}{1},\frac{0}{1},\frac{0}{1}) \ar[rr,"\preceq_{exch_{13}}"] & & (\frac{0}{1},\frac{0}{1},\frac{1}{1}) \\
  & \ar[from=ul,"\preceq_{exch_{12}}"'] (\frac{0}{1},\frac{1}{1},\frac{0}{1}) \ar[ur,"\preceq_{exch_{23}}"']
\end{tikzcd}
\quad\text{yet}\quad
\begin{tikzcd}[row sep=large,column sep=small]
  (\frac{1}{1},\frac{0}{0},\frac{0}{1}) \ar[rr,"\preceq_{exch_{13}}"] & & (\frac{0}{1},\frac{0}{0},\frac{1}{1}) \\
  & \ar[from=ul,dashed,"\preceq_{exch_{12}}"'] (\frac{0}{1},\frac{0}{-1},\frac{0}{1}) \ar[ur,dashed,"\preceq_{exch_{23}}"']
\end{tikzcd},
$$
illustrating that $q\preceq_{exch_{13}}q'' \centernot\implies \exists q' \ni q\preceq_{exch_{12}}q'\preceq_{exch_{23}}q''$ because we cannot `borrow' against a zero count.
\begin{defn}
  Let $\Delta Q^D = (Q^D-Q^D)/\!\triangleq$ denote equivalence classes $[q-q']$ of formal differences between $q,q'\in Q^D$ under the equivalence relation $(q_1-q_1') \triangleq (q_2-q_2') \iff q_1+q_2' = q_2+q_1'$.
\end{defn}

\begin{fact}
  $\Delta Q^D$ is obviously an Abelian group with the operation $(+)$ it inherits from $Q^D$, and hence a commutative ring over $\Z$ with multiplication $(\cdot)$ defined in the natural way.
\end{fact}

\begin{nota}
For each atomic arrow $\preceq_a$ let us recognize the corresponding formal difference $a \in \Delta Q^D$:
$$
\begin{aligned}\label{atoms}
  \mathrm{tol}_1 &= [\langle\frac{0}{1}\rangle_1 - \langle\frac{0}{0}\rangle] \\
  \mathrm{titro}_j &= [\langle\frac{0}{1}\rangle_j - \langle\frac{0}{1}\rangle_{j-1}] \\
  \mathrm{titrx}_j &= [\langle\frac{1}{1}\rangle_{j+1} - \langle\frac{1}{1}\rangle_j] \\
  \mathrm{det}_D &= [\langle\frac{0}{0}\rangle - \langle\frac{1}{1}\rangle_D] \\
  \mathrm{bal}_r &= [\langle\frac{0}{0}\rangle - \langle\frac{1}{1+r}\rangle_D] \\
  \mathrm{exch}_{j,k} &= [\langle\frac{0}{1},\frac{1}{1}\rangle_{j,k} - \langle\frac{1}{1},\frac{0}{1}\rangle_{j,k}],
\end{aligned}
$$
and denote the implied embedding of the relation $\preceq_r$ in $\Delta Q^D$ as $\preceq_r \xhookrightarrow{\phi} \Delta Q^D$.
\end{nota}

\begin{fact}\label{titrxelim}
  $\mathrm{titrx}_j = \mathrm{exch}_{j,j+1} + \mathrm{titro}_{j+1}$.  
\end{fact}

\begin{fact}\label{detelim}
  $\mathrm{det}_D = (\mathrm{tol}_1 + \mathrm{titro}_2 + ... + \mathrm{titro}_D)\cdot r + \mathrm{bal}_r$.  
\end{fact}

\begin{fact}\label{excheqs}
  $\mathrm{exch}_{jk} + \mathrm{exch}_{k\ell} = \mathrm{exch}_{j\ell}$.
  \begin{proof}
    The freedom to choose class representatives affords us the necessary `license to borrow'.
  \end{proof}
\end{fact}

\begin{fact}
  The embedding $\preceq_r \xhookrightarrow{\phi} \Delta Q^D$ is a monoid homomorphism.
\end{fact}

\begin{nota}\label{zeroed}
  For $q\in\Q$, we write $[q-\langle\frac{0}{0}\rangle]$ simply as $[q]$, and likewise $[\langle\frac{0}{0}\rangle-q]$ as $[-q]$.
\end{nota}

\begin{fact}
  Under Notation~\ref{zeroed}, $[q'-q] = [q']-[q]$ and $[q']-[q]=[\langle\frac{0}{0}\rangle] \iff q'=q$.
\end{fact}

\begin{fact}\label{coefs}
  For any $[\Delta q]\in \Delta Q^D$, we have
  \begin{equation}\label{characterize}
    [\Delta q] = \eta_1\cdot\mathrm{tol}_1 + \sum_{d=2}^D\eta_d\cdot\mathrm{titro}_d + \sum_{d=1}^{D-1}\gamma_d\cdot\mathrm{exch}_{d,d+1} + \gamma_D\cdot\mathrm{bal_r}
  \end{equation}
  for a unique vector $(\mathbf{\eta},\mathbf{\gamma}) = (\eta_1, ..., \eta_D, \gamma_1, ..., \gamma_D) \in \Z^{2D}$.
\end{fact}
\begin{proof}
  Write $[\Delta q] = [\frac{\Delta t}{\Delta n}]$, with $\Delta t = (t_1,...,t_D)$ and $\Delta n = (n_1,...,n_D)$, both in $\Z^D$.  Then \eqref{characterize} is equivalent to the following linear recurrence relations, easily solved in sequence for unique $\gamma_1,...,\gamma_D,\eta_1,...,\eta_D$:
$$
  \begin{aligned}
    \Delta t_1 &= - \gamma_1 \\
    \quad \Delta t_d &= \gamma_{d-1} - \gamma_d,\; d\in\{2,...,D\} \\
    \sum_d\Delta n_d &= \eta_1 - (1+r)\gamma_D \\
    \Delta n_d &= \eta_d-\eta_{d+1},\; d\in\{1,...,D-1\}.
  \end{aligned}
$$
\end{proof}

\begin{nota}
  In view of the embedding $\Delta Q^D \xhookrightarrow{(\eta,\gamma)} \Z^{2D}$ implied by Fact~\ref{coefs}, we will treat $(\eta,\gamma)$ as an alternative representation of $\Delta Q^D$, writing $[\Delta q] = (\eta,\gamma)$, $[\Delta q'] = (\eta',\gamma')$, and so forth.
\end{nota}

\begin{thm}\label{characterization}
  $q \preceq_r q' \iff [q'-q] = (\eta,\gamma) \in \N^{2D}$.
\end{thm}
\begin{proof}
  $(\implies)$ The RHS of \eqref{characterize} merely collects terms in the general element of $\phi(\preceq_r)$, eliminating $\mathrm{titrx}_*$ and $\mathrm{det}_D$ by Facts~\ref{titrxelim} and \ref{detelim}, then transforming the generic upper-triangular sum $\sum_{j<k}\gamma_{jk}\cdot\mathrm{exch}_{j,k}$ to a tidy superdiagonal form via Fact~\ref{excheqs}.

  $(\impliedby)$ By definition, $q\preceq_rq' \iff q = q_0 \preceq_{a_1} q_1 \preceq_{a_2} q_2\,\dots \preceq_{a_n} q_n = q'$ for some sequence of atomic arrows $(a_i)_{i=1}^n \in \{\mathrm{tol}_1,\mathrm{titro}_j,\mathrm{titrx}_j,\mathrm{exch}_{j,k},\mathrm{bal}_r\}$ and tallies $q_i\in Q^D$.  So the issue here becomes whether the terms collected in the {\em formal} sum on the RHS of \eqref{characterize} may be separated and transformed into such a sequence, with every partial sum (working left-to-right) constituting a valid tally:
  \begin{equation}\label{partials}
    [q_\ell] = [q_0] + \sum_{i=1}^\ell a_i \in Q^D,\,1\le\ell\le n.
  \end{equation}
  Now \textsc{wlog} we may safely permute the `purely additive' $\mathrm{tol}_1$ to the front of any such sum, and may delay the `purely subtractive' $\mathrm{bal}_r$ until the end.  Thus, we may deal with the narrower question whether the middle terms of \eqref{characterize} may always be spanned by a sequence like \eqref{partials}.  Since these middle terms {\em conserve} the sums $\sum t_d$ and $\sum u_d$, they implement a permutation.  Furthermore, this permutation can only preserve or increase net dose-intensity $N$.  So we need only show that our atomic-arrow repertoire suffices to construct {\em any} permutation having this property.  But this is straightforward: simply apply $\mathrm{titro}_*$ and then $\mathrm{titrx}_*$ as needed to increase $N$ stepwise up to $N'$, and then freely permute via $\mathrm{exch}_*$ to obtain $q'$.
\end{proof}

\begin{cor}\label{cancellation}
  `Cancellation': $q+b\preceq_r q'+b \implies q\preceq_r q'$.
\end{cor}

\begin{cor}\label{partord}
  $\preceq_r$ is in fact a \underline{partial order} on $Q^D$, since $q \cong q'$ requires both $(\gamma,\eta)$ and $(-\gamma,-\eta)$ to be non-negative, which can hold only if $\gamma = \eta = 0$, whence $t_d\equiv t_d'$, $u_d\equiv u_d'$ and thus $q=q'$.
\end{cor}

\begin{cor}\label{meetjoin}
  For $q,q'\in\Q_r$, with $[q]=(\eta,\gamma), [q']=(\eta',\gamma')$, then
  $$
  (\eta,\gamma)\wedge(\eta',\gamma') = (\eta\wedge\eta',\gamma\wedge\gamma')
  \quad\text{and}\quad
  (\eta,\gamma)\vee(\eta',\gamma') = (\eta\vee\eta',\gamma\vee\gamma').
  $$
  yield the meet $q\wedge q'$ and join $q\vee q'$, respectively, \underline{provided they correspond to valid tallies}.
\end{cor}

The need for the proviso in Corollary~\ref{meetjoin} is illustrated by the following diagram in $\Q_1$:

$$
\begin{tikzcd}[row sep=normal,column sep=small]
  (\frac{1}{2},\frac{0}{1}) & & (\frac{0}{1},\frac{1}{1}) \\
  (\frac{1}{3},\frac{0}{0}) \ar[u,"titro_2"] & (\frac{1}{1},\frac{0}{1}) \ar[ul,"tol_1"'] \ar[ur,"exch_{12}"] \\
  & (\frac{1}{2},\frac{0}{0}) \ar[ul,"tol_1"] \ar[u,"titro_2"'] \ar[uur,bend right,"titrx_1"']
\end{tikzcd}
\;\equiv\;
\begin{tikzcd}[row sep=normal,column sep=small]
  (\texttt{ox},\texttt{o}) & & (\texttt{o},\texttt{x}) \\
  (\texttt{oxo},-) \ar[u,"titro_2"] & (\texttt{x},\texttt{o}) \ar[ul,"tol_1"'] \ar[ur,"exch_{12}"] \\
  & (\texttt{ox},-) \ar[ul,"tol_1"] \ar[u,"titro_2"'] \ar[uur,bend right,"titrx_1"']
\end{tikzcd}
\xlongrightarrow{\ref{cancellation}}
\begin{tikzcd}[row sep=normal,column sep=small]
  (\texttt{x},\texttt{o}) & & (-,\texttt{x}) \\
  (\texttt{xo},-) \ar[u,"titro_2"] \\
  & (\texttt{x},-) \ar[ul,"tol_1"] \ar[uur,bend right,"titrx_1"']
\end{tikzcd},
$$
in which the `\texttt{xo}' notation allows us to see immediately that cancellation of an `\texttt{o}' at dose 1 renders the would-be meet $(\frac{1}{0},\frac{0}{1})$ calculated via Corollary~\ref{meetjoin} invalid.  The proviso is necessary for joins, as well:
$$
\begin{tikzcd}[row sep=normal,column sep=small]
  & (\frac{0}{0},\frac{1}{2}) \\
  & (\frac{0}{1},\frac{1}{1}) \ar[u,"titro_2"] \\
  (\frac{1}{1},\frac{0}{1}) \ar[ur,"exch_{12}"'] \ar[uur,bend left,"titrx_1"']
  & & (\frac{0}{0},\frac{2}{3}) \ar[ul,"bal_1"] \ar[uul,bend right,"det_2"]
\end{tikzcd}
\;\equiv\;
\begin{tikzcd}[row sep=normal,column sep=small]
  & (-,\texttt{ox}) \\
  & (\texttt{o},\texttt{x}) \ar[u,"titro_2"] \\
  (\texttt{x},\texttt{o}) \ar[ur,"exch_{12}"'] \ar[uur,bend left,"titrx_1"']
  & & (-,\texttt{xox}) \ar[ul,"bal_1"] \ar[uul,bend right,"det_2"]
\end{tikzcd}
\xlongrightarrow{\ref{cancellation}}
\begin{tikzcd}[row sep=normal,column sep=small]
  & (-,\texttt{x}) \\
  \\
  (\texttt{x},-) \ar[uur,bend left,"titrx_1"']
  & & (-,\texttt{xx}) \ar[uul,bend right,"det_2"]
\end{tikzcd},
$$

\begin{nota}
  Corollary~\ref{partord} licenses the notation $\prec$ defined by,
  $$
  q_1 \prec q_2 \iff q_1 \preceq q_2 \;\;\text{and}\;\; q_1 \neq q_2.
  $$
\end{nota}

\begin{fact}\label{tol-d}
  $\langle\frac{0}{0}\rangle \preceq_r \langle\frac{0}{1}\rangle_d \;\forall d\in 1..D$.
\end{fact}
\begin{proof}
  $$
  \langle\frac{0}{0}\rangle \preceq_{tol_1} \langle\frac{0}{1}\rangle_1 \preceq_{titr_2} ...  \preceq_{titr_d} \langle\frac{0}{1}\rangle_d.
  $$
\end{proof}

\begin{fact}\label{d11}
  $\langle\frac{1}{2}\rangle_d \preceq_r \langle\frac{0}{0}\rangle \;\forall d\in 1..D$.
\end{fact}
\begin{proof}
  From $\langle\frac{1}{2}\rangle_D \preceq_{bal_r} \langle\frac{0}{0}\rangle$, we proceed by induction on $d<D$:
  $$
  \langle\frac{1}{2}\rangle_d \preceq_{titr_d} \langle\frac{1}{1},\frac{0}{1}\rangle_{d,d+1} \preceq_{exch_{d,d+1}} \langle\frac{0}{1},\frac{1}{1}\rangle_{d,d+1} \preceq_{titr_d} \langle\frac{1}{2}\rangle_{d+1}.
  $$
\end{proof}

\begin{fact}\label{tox-d}
  $\langle\frac{1}{1}\rangle_d \preceq_r \langle\frac{0}{0}\rangle \;\forall d\in 1..D$.
\end{fact}
\begin{proof}
  $$
  \langle\frac{1}{1}\rangle_d \preceq_\text{Fact~\ref{tol-d}} \langle\frac{1}{2}\rangle_d \preceq_\text{Fact~\ref{d11}} \langle\frac{0}{0}\rangle.
  $$
\end{proof}

Facts~\ref{tol-d} and \ref{tox-d} reassure us that Definition~\ref{lesaf^D} suffices to obtain intuitively necessary evident-safety relations, such that each new observation of tolerability at any dose yields a safer tally, and each new observation of a toxicity yields a less-safe tally.  Note also how Facts~\ref{tol-d} and \ref{d11} have the similar effect of showing that the `edge-case' arrows $\preceq_{tol_1}$ and $\preceq_{bal_r}$ apply not just at $d=1$ and $d=D$, respectively, but indeed `homogeneously' across all doses.

\subsection{Application to the $\mbox{3+3}$ protocol}

Let us use $\preceq_r$ to examine the implicit pharmacology of the $\mbox{3+3}$ trial.  The smallest nontrivial $\mbox{3+3}$ design considers $D = 2$ doses, and has 46 possible paths through 42 accessible tallies, each with a dose-level recommendation in $\{0, 1, 2\}$ defined by the protocol \cite{norris-executable-2024}.  The Hasse diagram in Figure~\ref{fig-hasse2}(a) depicts the transitive reduction of the partial order $\preceq_1$ on these accessible tallies.  Coloring the tallies according to their dose recommendations allows us to see that these recommendations are mostly dose-monotone.  But careful examination reveals 2 exceptions, including one in which the dose recommendations for {\em final} tallies $(\frac{1}{6},\frac{1}{6}) \preceq (\frac{0}{6},\frac{2}{6})$ conflict with their evident safety.  Apparently, it is specifically the $\preceq_{exch}$ principle that \mbox{3+3} violates:
$$
(\frac{1}{6},\frac{1}{6}) = (\frac{0}{5},\frac{1}{5}) + (\frac{1}{1},\frac{0}{1}) \preceq_{exch_{12}} (\frac{0}{1},\frac{1}{1}) + (\frac{0}{5},\frac{1}{5}) = (\frac{0}{6},\frac{2}{6}).
$$
This non-monotonicity turns out in fact to be a general flaw in the \mbox{3+3} design for all $D>1$, which remarkably appears to have escaped notice even amid decades of severe criticism of this design by statisticians.  In Section~\ref{sec-dep}, we deal with the `rectification' of this flaw via Equation \eqref{rectification}.

The motivation for enlarging $\preceq_1$ generally to $\preceq_r$ may be seen in Figure~\ref{fig-hasse2}(b), depicting these same 42 accessible tallies partially ordered by $\preceq_2$.  Clearly, a much simpler transitive reduction is accomplished here, and this occurs without introducing any further non-monotonicities beyond the two already noted.  Thus, we may suppose that $\preceq_2$ more fully embraces whatever pharmacologic intuition is manifested in the \mbox{3+3} design.

\begin{figure}
  \includegraphics[width=6.5in]{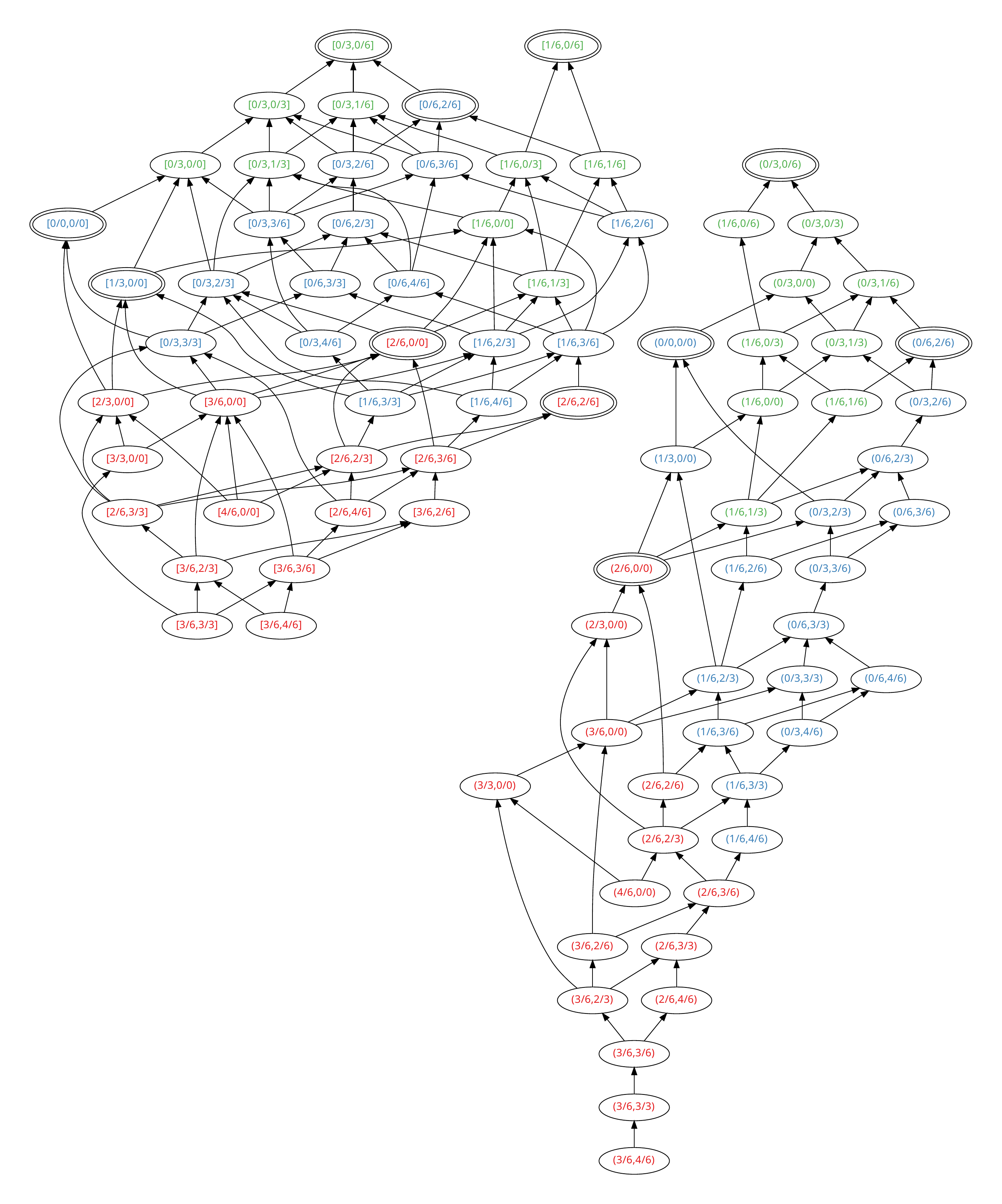}
  \caption{Hasse diagrams for the 42 tallies accessible to the 2-dose \mbox{3+3} protocol, partially ordered according to (a) $\preceq_1$ [top left] or (b) $\preceq_2$ [right], and colored according to the dose recommendation: 0=red, 1=blue, 2=green.  Doubled borders indicate the maximal elements of each colored subset.}\label{fig-hasse2}
\end{figure}

\section{Dose-Escalation Protocols}\label{sec-dep}

In this section, we omit the $r$ subscript from $\Q$ and $\preceq$, supposing $r \in \N$ arbitrarily fixed.

\begin{nota}
  Let $\D$ denote the category freely generated by the graph, $0 \rightarrow 1 \rightarrow \dots \rightarrow D$.
\end{nota}

\begin{defn}
  An \textbf{incremental enrollment [IE]} is a functor $\Q \xrightarrow{\;\;E\;\;} \D$.  Here functoriality imposes the core intuition of dose-escalation,
  $$
  q \preceq q' \implies E q \le E q',
  $$
  that dose assignment should correlate with evident safety.
\end{defn}

The color-coding of Figure~\ref{fig-hasse2} exhibits a {\em partial} function $Q^2 \xrightharpoonup{F} \{0,1,2\}$ defined by the $\mbox{3+3}$ design, mapping the subset $|\A| \subset Q^2$ of 42 tallies {\em accessible within the protocol} to their respective dose recommendations.  But the design yields no explicit dose assignment for tallies that are inaccessible to its rigid protocol, such as tallies with any denominator not a multiple of 3.  So, to release the operational constraint imposed by the \mbox{3+3} protocol's 3-at-a-time enrollment, we might like to pose and solve the extension problem,
\begin{equation}\label{extension-problem}
  \begin{tikzcd}[row sep=large]
    \A \ar[dr,hook,"\iota"] \ar[rr,"F"] & & \D \\
    & \Q\, \ar[ur,dashed,"E?"']
  \end{tikzcd}
\end{equation}
In order to regard \eqref{extension-problem} as a diagram in \textbf{Poset}, let us suppose---\textsc{wlog}, as we shall see---that $F$ has been rendered monotone by a `rectification' transformation such as,
\begin{equation}\label{rectification}
\bar{F}(a') = \bigwedge \{ F(a) \mid a' \preceq a \in \A \}.
\end{equation}

Now \eqref{extension-problem} looks like a typical set-up for seeking a Kan extension of $F$ along the inclusion functor $\iota$:
\begin{equation}
  \begin{tikzcd}[row sep=large]
    \A \ar[dr,hook,"\iota"'] \ar[rr,"F"{name=F},""{name=U,below}] & & \D \\
    & \Q\, \ar[ur,dashed,"\text{Lan}_\iota F"']
      \ar[Rightarrow,shorten=2.5mm,from=U,"\eta"]
  \end{tikzcd}
  \quad\quad\text{or}\quad\quad
  \begin{tikzcd}[row sep=large]
    \A \ar[dr,hook,"\iota"'] \ar[rr,"F"{name=F},""{name=U,below}] & & \D \\
    & \Q\, \ar[ur,dashed,"\text{Ran}_\iota F"']
      \ar[Rightarrow,shorten=2.5mm,to=U,"\epsilon"']
  \end{tikzcd}
\end{equation}
Of these, it is evidently the {\em right} Kan extension that approximates $F$ from the side of {\em safety}, since under $\epsilon:\text{Ran}_\iota F\cdot\iota \Rightarrow F$ we are assured of $\text{Ran}_\iota F(a)\le F(a)\,\forall a\in\A$.  Observe indeed that the explicit formula obtained from \cite{riehl-category-2016} Theorem 6.2.1 simply extends \eqref{rectification} to all of $\Q\supset\A$:
\begin{align}\label{right-Kan}
  \text{Ran}_\iota F(q) &= \lim(q\downarrow\iota\xrightarrow{\Pi_q}\A\xrightarrow{F}\D) \quad\quad \text{cf. \cite{riehl-category-2016} Eq. (6.2.3)} \notag \\
  &= \bigwedge\{F(a)\mid q\preceq a\in\A\}.
\end{align}
For $d < D$ we can express \eqref{right-Kan} equivalently by,
\begin{align}\label{almost-adjunction}
  \text{Ran}_\iota F(q) \le d &\iff \exists\,a\in\A \text{ such that } q\preceq a \text{ and } F(a)\le d \notag \\
  &\iff q\in\;\downarrow F^{-1}(\downarrow d),
\end{align}
whence
$$
\begin{aligned}
  \text{Ran}_\iota F(q) = d &\iff q\in\;\downarrow F^{-1}(\downarrow d) \smallsetminus \downarrow F^{-1}(\downarrow(d-1)) \\
  &\iff q\in\;\downarrow \bigcup_{j\le d} F^{-1}(j) \smallsetminus \downarrow \bigcup_{j<d} F^{-1}(j) \\
  &\iff q\in\;\bigcup_{j\le d}\downarrow F^{-1}(j) \smallsetminus \bigcup_{j<d}\downarrow F^{-1}(j) \\
  &\iff q\in\;\downarrow F^{-1}(d) \smallsetminus \bigcup_{j<d}\downarrow F^{-1}(j) \\
  &\iff q\in\;\downarrow\text{Max}(F^{-1}(d)) \smallsetminus \bigcup_{j<d}\downarrow\text{Max}(F^{-1}(j)),
\end{aligned}
$$
showing how $\text{Ran}_\iota F$ may be computed from the maximal elements of the fibers $F^{-1}(d)$:
\begin{equation}
\text{Ran}_\iota F(q) = \begin{cases}\label{Kan-cascade}
  0 &:\: q \,\in\, \downarrow\text{Max}(F^{-1}(0)) \\
  1 &:\: q \,\in\, \downarrow\text{Max}(F^{-1}(1)) \,\smallsetminus \downarrow\text{Max}(F^{-1}(0)) \\
   ... \\
  D-1 &:\: q \,\in\, \downarrow\text{Max}(F^{-1}(D-1)) \,\smallsetminus \bigcup_{j<D-1} \downarrow\text{Max}(F^{-1}(j)) \\
  D &:\: q \,\in\, \Q \,\smallsetminus \bigcup_{j<D}\downarrow\text{Max}(F^{-1}(j)).
\end{cases}
\end{equation}
Observe that the coloring in Figure~\ref{fig-hasse2} depicts the fibers of $F$ over $\D$, and that each fiber has just a few maximal elements.

If in \eqref{almost-adjunction} we were to replace the sets $F^{-1}(\downarrow d)$ by {\em single tallies} $G(d) = \bigvee F^{-1}(\downarrow d)$, we would obtain an IE that is strictly safer than $\text{Ran}_\iota F$, and left-adjoint to $G:\D\rightarrow\Q$.  This motivates the following
\begin{defn}
  A \textbf{lower Galois enrollment}\footnote{An adjunction between preorders is called a Galois connection, hence the name.} is an IE $\IE$ for which a right (upper) adjoint exists:
  \begin{center}
    \begin{tikzcd}
      \Q \ar[rrr,"E"{name=E,description}]
      & & & \D
      \ar[lll,bend right,"G"'{name=G}]
      \arrow[phantom,from=G,to=E,"\dashv" rotate=90]
    \end{tikzcd},
  \end{center}
providing the dose-assignment rule,
$$
E(q) \le d \quad \iff \quad q \preceq G(d).
$$
\end{defn}
One appeal of a Galois enrollment is that it yields a simple rule parametrized by $D$ tallies.  Writing $G(d) = g_d$, we have parameters $\{g_0\preceq ...\preceq g_{D-1}\} \subset \Q$ defining a lower-Galois enrollment by a cascading partition of $\Q$:
\begin{equation}
E(q) = \begin{cases}\label{cascade}
  0 &:\: q \,\in\, \downarrow\!g_0 \\
  1 &:\: q \,\in\, \downarrow\!g_1 \,\smallsetminus \downarrow\!g_0 \\
   ... \\
  D-1 &:\: q \,\in\, \downarrow\!g_{D-1} \,\smallsetminus \downarrow\!g_{D-2} \\
  D &:\: q \,\in\, \Q \,\smallsetminus \downarrow\!g_{D-1}.
\end{cases}
\end{equation}

While the Kan extension \eqref{Kan-cascade} does nicely motivate the lower-Galois enrollment, one may instead pass directly from the dose recommendations $\A\!\xrightarrow{F}\D$ of some given trial to a lower-Galois approximation.  Wishing to proceed {\em cautiously} in approximating $F$, we must ensure $E(q) \le F(q) \;\forall q \in \A$.  For a lower-Galois approximation $E\dashv (g_d)$, this cautionary requirement imposes a {\em lower bound} on its upper adjoint:
$$
F(q) \le d \; \Longrightarrow \; q \preceq g_d \quad \forall q \in \A,\; d \in \D. \label{gconstraint}
$$
A {\em closest} approximation will be had with minimal such $g_d$'s, easily obtained as the joins,
\begin{equation}\label{g-via-join}
  g_d = \bigvee F^{-1}(d),
\end{equation}
which again (like $\text{Ran}_\iota F$) render a preliminary `rectification' step \eqref{rectification} superfluous.  In our application below, the joins \eqref{g-via-join} are all readily obtainable via Corollary~\ref{meetjoin}.

\section{Trial Simulation}\label{sec-sim}

In this section, we present preliminary simulation results demonstrating the
feasibility of extending discrete-time dose-escalation designs to incorporate
discretionary titration either via Equation~\eqref{Kan-cascade} or via Equations \eqref{cascade}--\eqref{g-via-join}.
The Prolog code implementing these simulations is available in the public
repository \url{https://codeberg.org/dcnorris/DEDUCTION}.  This code requires
several numerical special functions and related probability distributions
implemented in a fork of Scryer Prolog available from
\url{https://github.com/dcnorris/scryer-prolog/tree/special}, features intended
for eventual inclusion in Scryer Prolog.  Appendix~\ref{rolling-rules} details
the definite clause grammar (DCG) \cite{iso-prolog-dcg-2025} that describes the operation of a simulated
trial in the continuous-time setting outlined in Section~\ref{ops}.

Working in $\Q_2 = (Q^D,\preceq_2)$, we obtain for the $D=3$ case of the \mbox{3+3} design the fiber maximal elements via
\begin{lstlisting}
?- d_fiberscolumn(3, FMColumn).
   FMColumn = [0-[[2/6,0/0,0/0]],
               1-[[0/0,0/0,0/0],[0/6,2/6,0/0]],
               2-[[0/3,0/0,0/0],[0/3,0/6,2/6]]].
\end{lstlisting}
and the right adjoint $G$ via
\begin{lstlisting}
?- d_joinscascade(3, Gs).
   Gs = [[0/3,0/6,0/0],[0/6,0/0,0/0],[2/6,0/0,0/0]].
\end{lstlisting}

We posit a simulation scenario in which the toxicity threshold (maximum tolerated dose, MTD) is distributed lognormally in the population with median $\mu$ and a biologically modest standard deviation,
$$
\ln \mathrm{MTD} \sim \mathrm{Normal}(\ln \mu, \ln 1.5).
$$
We suppose that our three doses $(x_1, x_2, x_3)$ are prespecified in geometric sequence with ratio 1.4,\footnote{This matches the 40\% dose-step increments of \cite{simon-accelerated-1997}.} and with the highest dose $x_3 = \mu$ happening to coincide with median MTD.  For simulation purposes, it is most convenient to draw pseudorandom MTDs on the same logarithmic scale with the dose levels $d \in \{1,2,3\}$ themselves:
$$
\mathrm{MTD} \sim \mathrm{Normal}(3, \frac{\ln 1.5}{\ln 1.4}).
$$
The probabilities of toxicity at each of the 3 dose levels are then
\begin{lstlisting}
> pnorm(1:3, mean=3.0, sd=log(1.5)/log(1.4)) # R code
[1] 0.04848889 0.20331388 0.50000000
\end{lstlisting}

Simulating Poisson arrivals at rate 2.5 per toxicity-assessment period, and enrolling 40 participants, 1000 independent realizations of our right Kan and lower-Galois extended trial designs yield final dose recommendations with probabilities tabulated below.  These are contrasted with the corresponding probabilities for the standard \mbox{3+3} design, calculated via \cite[Eqs (3--4)]{norris-what-2020} in Appendix~\ref{exactly-33}.
\begin{center}
\begin{tabular}{rcccc}
  Final dose recommendation & 0 & 1 & 2 & 3 \\
  \hline
  right Kan extension & 0.430 & 0.457 & 0.091 & 0.022 \\
  lower-Galois extension & 0.460 & 0.432 & 0.084 & 0.024 \\
  standard 3+3 design & 0.027 & 0.336 & 0.562 & 0.075
\end{tabular}
\end{center}
Being by construction strictly safer than the \mbox{3+3} design, our right Kan and lower-Galois approximations of course yield more cautious recommendations.

\section{Future Work}

Comprehensive sets of such simulation experiments could help orient oncology clinical trialists to various frequentist characteristics of this new design, such as its {\em target toxicity probability}, a commonly discussed design parameter for dose-escalation trials.  But to exhibit the genuinely new characteristics of these dose-titration designs---such as their benefits for individual trial participants, or the fuller picture they yield of the population distribution of MTD---will require developing dynamic, interactive data visualizations \cite{norris-patient-centered-2020}.

While the long dominance and universal familiarity of the \mbox{3+3} design have made it an obligatory first target for our right Kan and lower-Galois approximations, several classes of newer parametric and semiparametric designs \cite{clertant-semiparametric-2017}, including CRM \cite{oquigley-continual-1990} and BOIN \cite{yan-boin-2020}, may present more interesting targets.  Alternatively, if dose-titration designs could gain acceptance on their own merits---that is, apart from their relations to more familiar dose-escalation designs---this would open up interesting possibilities for computationally challenging searches over discrete spaces of lower-Galois enrollment functors, to identify {\em de novo} designs with specified safety properties.

\clearpage
\bibliographystyle{unsrturl}
\bibliography{escatology}

\clearpage
\appendix

\section{DCG Simulation of Rolling Enrollment with Titration}\label{rolling-rules}

Definite clause grammar (DCG) \texttt{rolling//4} describes a list of events
occurring in a rolling-enrollment trial.  Its rules have 4 arguments:
\begin{enumerate}
\item[\texttt{E\_2}] a binary predicate defining the dose-recommendation rule
\item[\texttt{Q}] a cumulative tally from assessments completed up to now
\item[\texttt{Ws}] a queue of patients waiting to enroll
\item[\texttt{As}] a \texttt{keysort/2}-ed list of \texttt{Time-A} pairs for future arrivals/assessments \texttt{A} of the form:
  \begin{itemize}
  \item \texttt{arr(MTD)}, arrival of patient with toxicity threshold \texttt{MTD} $\in \R^+$ on the dose-level scale
  \item \texttt{ao(Rx,MTD)} or \texttt{ax(Rx,MTD)}, tolerated and non-tolerated {\em enrolling} doses respectively
  \item \texttt{to(Rx,MTD)} or \texttt{tx(Rx,MTD)}, denoting likewise assessments at subsequent {\em titrated} doses.
  \end{itemize}
\end{enumerate}
We scale time so that the toxicity assessment period is 1.  This allows us
(among other conveniences) to model the time-to-toxicity in case MTD < Rx simply
as Delay = MTD/Rx.

Provided that the waiting queue is empty, a patient arriving when the current recommended dose is nonzero will be enrolled at that dose.  But a patient arriving at a time when current enrolling dose is 0 enters the waiting queue.
\begin{lstlisting}
rolling(E_2, Q, [], [Z-arr(MTD)|As]) --> { rec(E_2, Q, As, Rx), Rx > 0 },
                                         dose(enroll(Rx,MTD@Z), As, As1),
                                         rolling(E_2, Q, [], As1).
rolling(E_2, Q, Ws, [Z-arr(MTD)|As]) --> { rec(E_2, Q, As, 0) },
                                         enqueue(MTD@Z, Ws, Ws1),
                                         rolling(E_2, Q, Ws1, As).

rec(E_2, Q, As, Rx) :- tally_pending_pesstally(Q, As, Qp), call(E_2, Qp, Rx).

dose(Event, As, As1) --> { Event =.. [_, Rx, MTD@Z],
                           (   Rx >  MTD, A = ax(Rx,MTD), Za is Z + MTD/Rx
                           ;   Rx =< MTD, A = ao(Rx,MTD), Za is Z + 1.0
                           ),
                           sched(As, Za-A, As1) },
                         [Event].

sched(As, Za-A, As1) :- keysort([Za-A|As], As1).

enqueue(MTD@Z, Ws, Ws1) --> { append(Ws, [MTD], Ws1) }, [enqueue(MTD)@Z].
\end{lstlisting}
But whenever the current dose recommendation becomes nonzero, waiting participants receive their doses in order of arrival:
\begin{lstlisting}
rolling(E_2, Q, [MTD|Ws], [now(Z)|As]) --> { rec(E_2, Q, As, Rx), Rx > 0 },
                                           dose(dequeue(Rx,MTD@Z), As, As1),
                                           rolling(E_2, Q, Ws, [now(Z)|As1]).
\end{lstlisting}
Enrollment out of the waiting queue continues until the queue is empty, or the current recommended dose drops to 0:
\begin{lstlisting}
rolling(E_2, Q,     [], [now(_)|As]) --> rolling(E_2, Q, [], As).
rolling(E_2, Q, [W|Ws], [now(_)|As]) --> { rec(E_2, Q, As, 0) },
                                         rolling(E_2, Q, [W|Ws], As).
\end{lstlisting}
Tallying a tolerated dose, whether an enrolling dose \texttt{ao} or titrated
dose \texttt{to}, injects a {\em future} titration into \texttt{As}, unless
already at maximum dose.  Because tallying a tolerated dose may increase the
current enrolling dose---and in particular, increase it from 0 to a positive
dose level---we transiently substitute a term of the form \texttt{now(Time)} in
place of a just-tallied non-toxicity in \texttt{As}, to effect a `freeze-frame'
during which one or more enrollments may occur out of \texttt{Ws}.
(Cf. footnote 1 in the main text.)
\begin{lstlisting}
rolling(E_2, Q, Ws, [Z-O|As]) --> { O =.. [_o, Dose, MTD], member(_o, [ao,to]) },
                                  tallyo(_o, Q, Dose, Q1, MTD@Z),
                                  { length(Q, D) },
                                  d_updose(D, Dose, MTD@Z, As, As1),
                                  rolling(E_2, Q1, Ws, [now(Z)|As1]).

tallyo(ao, Q, Dose, Q1, MTD@Z) --> { tallyo(Q, Dose, Q1) }, [o(Dose,MTD)@Z].
tallyo(to, Q, Dose, Q1, MTD@Z) --> {  titro(Q, Dose, Q1) }, [o(Dose,MTD)@Z].

d_updose(D, D, _, As, As) --> [].
d_updose(D, Dose, MTD@Z, As, As1) --> { #Dose #< #D,
                                         #Rx #= #Dose + 1,
                                         titrwait(Wait),
                                         (   Rx >  MTD, A = tx(Rx,MTD),
                                             Z1 is Z + Wait + MTD/Rx
                                         ;   Rx =< MTD, A = to(Rx,MTD),
                                             Z1 is Z + Wait + 1
                                         ),
                                         sched(As, Z1-A, As1) },
                                       [updose(Rx,MTD)@Z1].
\end{lstlisting}
By specifying a delay before upward titration, we can effect a {\em gradual} introduction of titration.  (Setting this arbitrarily high would effectively eliminate titration from the protocol.)
\begin{lstlisting}
titrwait(1).
\end{lstlisting}
Tallying toxicities does not change the current dose recommendation, and so is
more straightforward.  (We need not explicitly model the dose reduction which
would ensue upon assessment of toxicity.)
\begin{lstlisting}
rolling(E_2, Q, Ws, [Z-X|As]) --> { X =.. [_x, Dose, MTD], member(_x, [ax,tx]) },
                                  tallyx(Q, Dose, Q1, MTD@Z),
                                  rolling(E_2, Q1, Ws, As).

tallyx(Q, Dose, Q1, MTD@Z) --> { tallyx(Q, Dose, Q1) }, [x(Dose,MTD)@Z].
\end{lstlisting}
Finally, when \texttt{As=[]}, no further arrivals or assessments are pending,
and the trial concludes emitting the final tally \texttt{Q} and its associated
dose recommendation \texttt{Rx}:
\begin{lstlisting}
rolling(E_2, Q,    [], []) --> { call(E_2, Q, Rx) }, [Q, next(Rx)].
rolling(E_2, Q, [_|_], []) --> { call(E_2, Q, 0)  }, [Q, next(0)].
\end{lstlisting}

Predicates \texttt{tallyo/3}, \texttt{titro/3} and \texttt{tallyx/3} are defined quite straightforwardly using declarative integer arithmetic in \url{https://codeberg.org/dcnorris/DEDUCTION/src/branch/main/tally.pl}.  The definition of \texttt{tally\_pending\_pesstally/3} may be found in \url{https://codeberg.org/dcnorris/DEDUCTION/src/branch/main/queueing.pl}.

\newpage
\section{Recommendation Probabilities for the \mbox{3+3} Design}\label{exactly-33}

Precise probabilities for outcomes of the standard \mbox{3+3} trial design on the simulation scenario of Section~\ref{sec-sim} may be obtained as follows, using R package {\bf precautionary} available from \url{https://github.com/dcnorris/precautionary} and documented at \url{https://dcnorris.github.io/precautionary/}.

\lstset{
  language=R
}

\begin{lstlisting}
library(precautionary) # install via remotes::install_github("dcnorris/precautionary")

finrec33 <- function(Tcd) {
  t <- apply(Tcd, 2, sum)
  for (d in ncol(Tcd):1)
    if (!is.na(t[d]) && t[d] < 2)
      return(d)
  return(0)
}

p <- pnorm(1:3, mean=3.0, sd=log(1.5)/log(1.4))
q <- 1 - p
pq <- c(p,q)

b <- precautionary:::b[[3]]
U <- precautionary:::U[[3]]
log_pi <- b + U %*% log(pq)

rx <- apply(precautionary:::T[[3]], 3, finrec33)

doses <- 0:3
names(doses) <- paste0("DL", 0:3)

rec_probs <- t(outer(doses, rx, "==") %*% exp(log_pi))  

> rec_probs
            DL0       DL1       DL2        DL3
[1,] 0.02710926 0.3361197 0.5619761 0.07479493
\end{lstlisting}

\end{document}